\documentclass [a4,10pt,twoside]{article}
\usepackage{amsmath, amssymb, amsthm, amsfonts, amsxtra, latexsym, amscd,
pb-diagram,  graphics,rotating}
\usepackage{portland,graphpap}

\def\v{\vskip}

\theoremstyle{plain}
\newtheorem{dl}{Theorem}[section]

\newtheorem{md}[dl]{Proposition}

\newtheorem{dn}[dl]{Definition}

\newcommand{\tx}{\otimes }
\newcommand{\ts}{\oplus}
\newcommand{\Ah}{\frak{A}}
\newcommand{\Lh}{\frak{L}}
\newcommand{\Rh}{\frak{R}}
\newcommand{\Ra}{\widehat{R}}
\newcommand{\Hh}{\mathcal{H}}
\newcommand{\A}{\mathcal{A}}
\newcommand{\Fm}{\widetilde{F}}

\newcommand{\La}{\widehat{L}}
\newcommand{\Fh}{\Acute {F}}

\newcommand{\Fc}{\Breve {F}}
\newcommand{\Lc}{\Breve {L}}

\begin{document}
\title{ THE COHERENCE THEOREM FOR ANN-CATEGORIES}
\author{Nguyen Tien Quang}
\pagestyle{myheadings} 
\markboth{The coherence theorem for Ann-categories}{Nguyen Tien Quang}
\maketitle
\setcounter{tocdepth}{1}

%\begin{center}
%{\bf  Nguyen Tien Quang}\\
%{Hanoi University of Education}
%\end{center}

\begin{abstract}%{Introduction}
\noindent This paper\footnote{This paper has been published (in Vietnamese) in Vietnam Journal of Mathematics Vol. XVI, No 1, 1988.} presents the proof of the coherence theorem for Ann-categories whose set of axioms and original basic properties were given in [9]. Let 
$$\A=(\A,{\Ah},c,(0,g,d),a,(1,l,r),{\Lh},{\Rh})$$ be an Ann-category. The coherence theorem  states that in the category $ \A$, any morphism built from the above isomorphisms and the identification by composition and the two operations $\tx$, $\ts$ only depends on its source and its target.\\
The first coherence theorems were built for monoidal and symmetric monoidal categories by Mac Lane [7]. After that, as shown in the References, there are many results relating to the coherence problem  for certain classes of categories.\\
For Ann-categories, applying Hoang Xuan Sinh's ideas used for Gr-categories in [2], the proof of the coherence theorem is constructed by faithfully ``embedding'' each arbitrary Ann-category into a quite strict Ann-category. Here, a {\it quite strict} Ann-categogy is an Ann-category whose all constraints are strict, except for the commutativity and left distributivity ones.\\
\noindent This paper is the work continuing from [9]. If there is no explanation, the terminologies and notations in this paper mean as in [9].
\end{abstract}
\v 0.5cm
\section {Canonical isomorphisms}
In this section, we define some canonical isomorphisms induced by isomorphisms 
$c,  \Lh,  \widehat{L}^A $ and the identification, laws $\tx, \ts$ on the quite strict Ann-category $\A.$\\
Let $I$ be a fully ordered limited set. If $I \not=\emptyset$ and $\alpha$ is the maximal of $I,$ we will denote $I' = I\setminus  \{ \alpha\};$ and the notation $|I|$ refers to the cardinal of $I.$

\begin{dn}$[1]$
The canonical sum $\sum\limits_I A_{i}$ where $A_{i} \in Ob\A, i \in I $ is defined inductively as follows \\
1. $\sum\limits_I  A_i = 0$ if $I = \emptyset$ and $\sum\limits_I A_{i} = A_{\alpha}$ if $I=\{\alpha\}$.\\
2. $\sum\limits_IA_{i}=(\sum\limits_{I'} A_i \oplus A_{\alpha}) $ if $| I | > 1$.
\end{dn}
\begin{dn}
We define the isomorphism 
$$ \nu_{\sum A_{i} , \sum B_{i}}: (\sum\limits_I A_{i}) \oplus (\sum\limits_I B_{i} ) \longrightarrow \sum\limits_I (A_{i}\oplus B_{i}),$$
(which is abbreviated by $\nu_{I} $) by induction on $|I|$ as follows\\
1. $\nu_{I} = id$   if  $| I | \leq 1$.\\
2. if $|I|>2,$ $\nu_{I}$ is defined by the folllowing commutative diagram
%%%%%%bieu do 1%%%%%%%%%%%%%%

\[\begin{CD}
(\sum\limits_I A_i)\ts (\sum\limits_I B_i)@>id>> (\sum\limits_{I'} A_i)\ts A_{\alpha} \ts (\sum\limits_{I'} B_i) \ts B_{\alpha}\\
@V \nu_I VV   @VV \nu V  \\
\sum\limits_I (A_i\ts B_i) @> \nu_{I'} \ts id_{A_{\alpha}\ts B_{\alpha}} >> (\sum\limits_{I'} A_i)\ts (\sum\limits_{I'} B_i)\ts A_{\alpha} \ts B_{\alpha}
\end{CD}\]
where $\nu = id \oplus c \oplus id$. We can see that the isomorphism $\nu_{I}$ is built only from isomorphisms $c, id$ by law $\oplus$. Moreover, the isomorphism $\nu_{I}$ is natural.
\end{dn}
\begin{dn}
We define following isomorphisms
 $$ u_{I, J}:\sum\limits_I \sum_{J}(A_{i}\otimes B_{j})\longrightarrow \sum_{J} \sum\limits_I(A_{i}\otimes B_{j})$$
by induction on $| I |$ as follows\\
1. $ u_{I, J} = id $ if $| I | \leq 1$ or $ J = \emptyset $.\\
2. $ u_{I, J} = \nu_{J}(u_{I', J}\oplus id)$ if $| I | > 1$.\\
So, isomorphisms $ u_{I, J}$ are also built from the isomorphisms $c, id$ by law $\oplus$ and these morphisms are functorial.
\end{dn}
\begin{dn}
We define following isomorphisms
$$ F_{I, J}: (\sum\limits_IA_{i})\otimes (\sum_{J} B_{j})\longrightarrow \sum\limits_I \sum_{J}(A_{i}\otimes B_{j}) = \sum_{I \times J}(A_{i}\otimes B_{j})$$
where $I\times J$  is odered alphabetically as follows\\
1. $ F_{I, J} = id : 0 \otimes (\sum_{J} B_{i})\rightarrow 0$ if $I= \emptyset.$\\
\noindent (since $\Rh = id$ we have $\widehat{R}^A=id$ for all $A \in Ob\A$)\\ 
2. $ F_{I, J} = \La^X : X \otimes 0 \rightarrow 0$ where $ X = \sum\limits_I A_{i}$ if $ J = \emptyset$.\\
3. $ F_{I, J} = \sum\limits_I f_{A_{i}}$ if $ I \not= \emptyset$ and $ J \not= \emptyset,$ where\\
$$ f_{A} : A \otimes (\sum_{J} B_{j}) \rightarrow \sum_{J}A\otimes B_{j}$$
is defined as follows: If  $|J|=1,$ $f_{A}=id;$ whereas $f_{A}$ is defined by induction on $|J|$ by the following commutative diagram
%%%%%%bieu do 2%%%%%%%%%%%%%%
\[\begin{CD}
 A\tx (\sum\limits_J B_j)@>f_A>> (\sum\limits_{J} (A\tx B_j)\\
@V \Lc^A VV                    @|  \\
A\tx (\sum\limits_J B_j)\ts (A\tx B_\beta) @>f'_A \tx id >> (\sum\limits_{J'} A\tx B_j)\ts (A\tx B_\beta)
\end{CD}\]
where $\beta$ is the maximal element of $ J$ and $ J' =J\setminus \{\beta\}$.
\end{dn}
\begin{dn}
We define following isomorphisms
$$ K_{I, J}: (\sum\limits_IA_{i})\otimes (\sum_{J}B_{j})\longrightarrow \sum\limits_J \sum_I(A_{i}\otimes B_{j}) $$
as follows\\
1. $ K_{I, J} = id$  if $ I = \emptyset$.\\
2. $ K_{I, J} = \La^X,$ where $ X = \sum\limits_I A_{i}$ if $ J = \emptyset$.\\
3. $ K_{I, J} = f_X,$ where $ X = \sum\limits_I A_{i}$ in other cases.
\end{dn}
Then, we have the following proposition immediately
\begin{md}
With canonical sums $\sum\limits_I A_i, \sum\limits_J A_j$, we have the relation
$$K_{I,J}=u_{I,J}\cdot F_{I,J}.$$
\end{md}
Applying this proposition we can prove 	 
\begin{md}
Assume $J_1, J_2$ be non-empty subsets of $J$ such that $J=J_1 \coprod J_2$ and $j_1<j_2$ if $j_1 \in J_1, j_2 \in J_2$. Then for sums $A= \sum\limits_I A_i$, $\sum\limits_J B_j$, $\sum\limits_{J_1} B_j$, $\sum\limits_{J_2} B_j$ we have following relations
\begin{align*}
F_{I,J}&=\nu_I \cdot (F_{I,J_1}\ts F_{I,J_2})\cdot \Lc^A\\
F_{J,I}&=F_{J_1,I}\ts F_{J_2,I}.
\end{align*}
\end{md}
We will give the proof of this proposition in detail to illustrate the proof using commutative diagrams. Hereafter, for convenience, we write $AB$ instead of $A\tx B$ for all $A,B\in Obs(\A).$\\
\begin{md}
In the Ann-category $\A,$ the following diagrams 
%%%%%%bieu do 3%%%%%%%%%%%%%%
\[\begin{CD}
 (\sum\limits_I A_i)\tx (\sum\limits_J B_j)\tx(\sum\limits_T C_t)@>F_{I,J}\tx id>> (\sum\limits_{I\times J} A_i\tx B_j)\tx(\sum\limits_T C_t)\\
@V id\tx F_{J,T}  VV   @VV F_{I\times J,T} V  \\
(\sum\limits_I A_i)\tx  (\sum\limits_{J\times T} B_j \tx C_t) @>>F_{I,J\times T}>\sum\limits_{I\times J\times T} A_i\tx B_j \tx C_t
\end{CD}\]
commute.
\end{md}
\vspace{-1cm}
\begin{proof}
1. In case $I=\emptyset,$ we have the proposition proved since the diagram (1.1) becomes the following one
\[\begin{CD}
 0 (\sum\limits_J B_j)(\sum\limits_T C_t)@>id\tx id>> 0(\sum\limits_T C_t)\\
@V id\tx F_{J,T}  VV   @VV id V  \\
0 (\sum\limits_{J\times T} B_j \tx C_t) @>>id >0
\end{CD}\]
whose commutativity follows from $\Ra^X=id$ and the property of the zero object (see Prop.3.2 [9]). In case $J=\emptyset$ or $T=\emptyset$ the proposition is proved similarly. Hence, we now can suppose that $I,J,T$ are all not empty.\\
2. In case $| I |=1$. Firstly, consider the case in which $|J|=1.$ We prove the proposition by induction on $|T|.$ Then\\
\indent If $|T|=1,$ the proof is obvious.\\
\indent If $|T|=2,$ the diagram commutes thanks to the axiom (1.1) of Ann-categories (see [9]).\\
\indent If $|T|>2,$ consider the diagram (1.2). 
\begin{center}
\setlength{\unitlength}{1cm}
\begin{picture}(15,9)

\put(4.5,1){$A( \sum\limits_{T'}BC_t \ts BC_\gamma)$}
\put(8.6,1){$A( \sum\limits_{T'}BC_t) \ts ABC_\gamma$}
\put(4.5,3){$A( \sum\limits_{T'}BC_t \ts BC_\gamma)$}
\put(8.4,3){$\sum\limits_{T'}ABC_t \ts ABC_\gamma$}
\put(5,5){$AB (\sum\limits_{T}C_t )$}
\put(8.2,5){$\sum\limits_{T'}ABC_t \ts ABC_\gamma$}
\put(5,6.5){$AB (\sum\limits_{T}C_t )$}
\put(8.2,6.5){$AB(\sum\limits_{T'}C_t)\ts ABC_\gamma$}
\put(4,8){$A(B\sum\limits_{T'}C_t\ts BC_\gamma)$}
\put(8.2,8){$AB (\sum\limits_{T}C_t )\ts ABC_\gamma$}

\put(2.3,4.2){$id\tx (f'_B \ts id)$}
\put(5.8,1.9){$id$}
\put(5.8,3.9){$id\tx f_B$}
\put(5.8,5.8){$id$}
\put(5.8,7.3){$id \tx \Lc^B$}
\put(7.7,1.3){$\Lc^A$}
\put(7.5,3.3){$f_A$}
\put(7,5.2){$f_{AB}$}
\put(7,6.7){$\Lc^{AB}$}
\put(7.2,8.3){$\Lc^A$}
\put(9.6,1.9){$f'_A\ts id$}
\put(9.6,3.9){$id$}
\put(9.6,5.8){$f_{AB}\ts id$}
\put(9.6,7.3){$id$}
\put(11.5,4.2){$(id\tx f'_B) \ts id$}

\put(5.6,1.5){\vector(0,1){1}}
\put(5.6,4.5){\vector(0,-1){1}}
\put(5.6,6.1){\vector(0,-1){0.6}}
\put(5.6,7.8){\vector(0,-1){0.8}}
\put(9.5,1.5){\vector(0,1){1}}
\put(9.5,4.5){\vector(0,-1){1}}
\put(9.5,6.3){\vector(0,-1){1}}
\put(9.5,7.8){\vector(0,-1){0.8}}
\put(7.4,1.1){\vector(1,0){1}}
\put(7.4,3.1){\vector(1,0){0.8}}
\put(6.7,5.1){\vector(1,0){1.2}}
\put(6.7,6.6){\vector(1,0){1.2}}
\put(7,8.1){\vector(1,0){1}}
%\put(7,12.1){\vector(1,0){1}}

\put(2,1.1){\line(1,0){2}}
\put(2,1.1){\line(0,1){7}}
\put(2,8.1){\line(1,0){1.7}}
\put(12,1.1){\line(1,0){2}}
\put(14,1.1){\line(0,1){7}}
\put(14,8.1){\line(-1,0){2}}

\put(14.5,5){(2)}
\end{picture}
\end{center}
In this diagram, the region (I) commutes thanks to the axiom (1.1) in [9]; regions (II), (IV), (V) commute thanks to definitions of isomorphisms $f_{AB}$, $f_A$, $f_B$; (VI) commute thanks to the inductive supposition; the parameter commutes since $\Lc^A$ is a functorial isomorphism. Therefore, the region (III) commutes. This completes the proof.\\
After that, still with the condition $|I|=1,$ we can prove the proposition with $|J|>1$ by induction on $|J|.$\\
3. Now if $|I|>1,$ consider the diagram (1.3). In this diagram, the region (I) commute since $\Rh=id$ is a functorial isomorphism; the region (II) commutes thanks to the inductive supposition for the first component of sums, for the second component since the case $|I|=1$ has just been aproved above; the region (III) commutes thanks to the property of the isomorphism $F_{I,J}$ (see the Prop.1.7); regions (IV), (V) commute thanks to definitions of $F_{I,J}$ and $F_{I,J\times T}.$ So the parameter commutes. This completes the proof.                             \end{proof}
\newpage
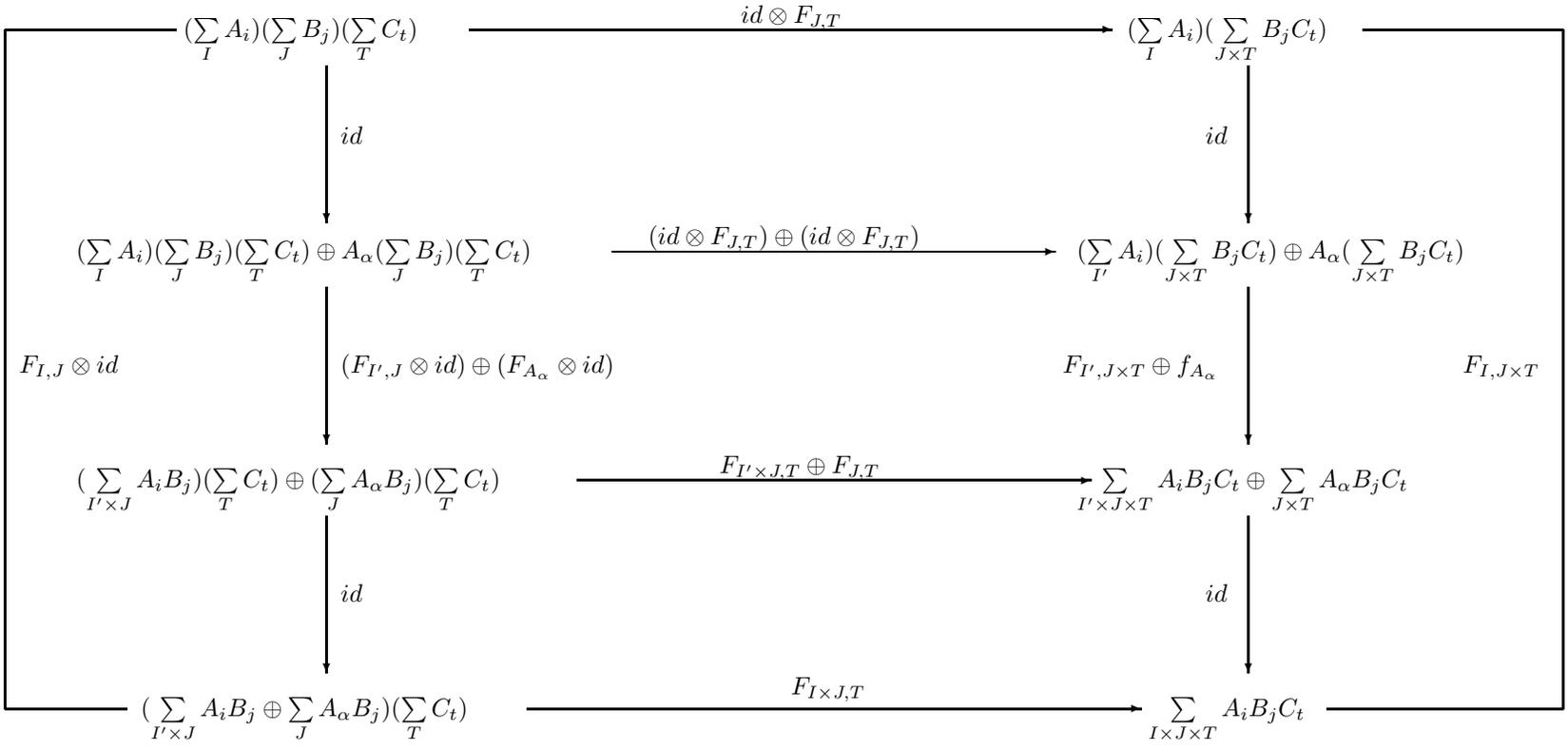
\begin{center}
\begin{sideways}
\setlength{\unitlength}{1cm}
\begin{picture}(21.8,9.5)

\put(1.9,0){$(\sum\limits_{I'\times J} A_iB_j \ts \sum\limits_J A_\alpha B_j)(\sum\limits_T C_t)$}
\put(16,0){$\sum\limits_{I\times J\times T} A_i B_j C_t$}
\put(1,3.2){$(\sum\limits_{I'\times J} A_iB_j)(\sum\limits_T C_t)\ts (\sum\limits_J A_\alpha B_j)(\sum\limits_T C_t)$}
\put(15,3.2){$\sum\limits_{I'\times J\times T} A_i B_j C_t \ts \sum\limits_{J\times T} A_\alpha B_j C_t$}
\put(1,6.4){$(\sum\limits_I A_i)(\sum\limits_J B_j)(\sum\limits_T C_t)\ts A_\alpha (\sum\limits_J B_j)(\sum\limits_T C_t)$}
\put(15,6.4){$(\sum\limits_{I'} A_i)(\sum\limits_{J\times T} B_j C_t) \ts A_\alpha (\sum\limits_{J\times T} B_j C_t)$}
\put(2.5,9.5){$(\sum\limits_I A_i)(\sum\limits_J B_j)(\sum\limits_T C_t)$}
\put(15.7,9.5){$(\sum\limits_I A_i)(\sum\limits_{J\times T} B_j C_t)$}
\put(11,0.3){$F_{I\times J, T}$}
\put(4.7,1.6){$id$}
\put(4.7,4.8){$(F_{I',J}\tx id)\ts (F_{A_\alpha}\tx id)$}
\put(4.7,8){$id$}
\put(0.2,4.8){$F_{I,J}\tx id$}
\put(10,3.4){$F_{I'\times J,T}\ts F_{J,T}$}
\put(9,6.6){$(id\tx F_{J,T})\ts (id\tx F_{J,T})$}
\put(10.3,9.7){$id\tx F_{J,T}$}
\put(16.8,1.6){$id$}
\put(14.8,4.8){$F_{I',J\times T}\ts f_{A_\alpha}$}
\put(16.8,8){$id$}
\put(20.4,4.8){$F_{I,J\times T}$}
\put(4.5,2.8){\vector(0,-1){2.2}}
\put(4.5,6){\vector(0,-1){2.2}}
\put(4.5,9.1){\vector(0,-1){2.2}}
\put(17.4,2.8){\vector(0,-1){2.2}}
\put(17.4,6){\vector(0,-1){2.2}}
\put(17.4,9.1){\vector(0,-1){2.2}}
\put(7.3,0.1){\vector(1,0){8.6}}
\put(8,3.3){\vector(1,0){7.2}}
\put(8.5,6.5){\vector(1,0){6.2}}
\put(6.5,9.6){\vector(1,0){9}}
\put(0,0.1){\line(1,0){1.7}}
\put(0,0.1){\line(0,1){9.5}}
\put(0,9.6){\line(1,0){2.4}}
\put(18.5,0.1){\line(1,0){3.3}}
\put(21.8,0.1){\line(0,1){9.5}}
\put(21.8,9.6){\line(-1,0){2.8}}

\end{picture}
\end{sideways}
\end{center}
\newpage
%%%%Chuong 2%%%%%%%%%%%%%%%%%%%%
\section{The coherence theorem for Ann-categories}
Let $\A$ be a quite strict Ann-category. Assume $X_s, s\in \Omega$ be a non-empty, limited family of objects of $\A$ and $Y$ be an expression of the family $X_s$ with operations $\tx$ and $\ts.$ With distributivity constraints $\Lh,\Rh =id,$ induced isomorphisms $\widehat{L}^A$ and isomorphisms $\Ah=id$, $g, d=id$, $a=id$, $l, r=id$, we can write $Y$ as a sum of {\it monomials} of objects of $X_s$ by using the following isomorphism
$$h: Y \rightarrow \sum\limits _I A_i$$
where $A_i \not=0$ for all $i$ if $I \not=\emptyset$ and $h$ is built from the identification, and isomorphisms $\Lh$, $\widehat{L}^A.$ A such a pair $(h,\sum\limits_I A_i)$ is called an \textit{expansion form} of $Y$. We now define \textit{a canonical expansion form} of $Y$ by induction on its length, where the length of expansion form $Y$ is the total number of times of appearances of objects $A_i$ in $Y.$ It is easily to see that any $Y$ whose length is more than 1 can be written in the form of $U\tx V$ or $U\ts V$. That implies 
\begin{dn}
The canonical expansion form 
$$h: Y \rightarrow \sum\limits _I A_i$$  
of $Y$ is defined as follows\\
1. If $Y=\sum\limits _I X_i,$ $h=id:\sum\limits _I X_i  \rightarrow \sum\limits _{I'} X_i$, where $I$ is a  subset of $\Omega$; whereas $I'$ is the set of indexes $i$ such that $X_i \not= 0$.\\ 
2. If $Y=U\tx V,$ the isomorphism $h$ is the composition
\[\begin{CD}
Y=U\tx V @>u\tx v>> (\sum\limits _J B_j)\tx(\sum\limits _T C_t)@>F_{J,T}>>\sum\limits _{J\times T} B_j
C_t
\end{CD}\]
where $(u,\sum\limits _J B_j)$, $(v,\sum\limits _T C_t)$ are, respectively, canonical expansion forms of $U$, $V$ and defined by the induction supposition.\\
3. If $Y=U_1\ts U_2,$ $h$ is the composition
\[\begin{CD}
Y=U_1\ts U_2 @>u_1\ts u_2>> (\sum\limits _{I_1} B_i)\tx(\sum\limits _{I_2} B_i)@>id>>\sum\limits _{I'} B_i
\end{CD}\]
where $u_1$, $u_2$ are defined by the inductive supposition; $I=I_1 \amalg I_2$; $i_1 < i_2$ if $i_1 \in I_1, i_2 \in I_2$ and $I'$ is the set of indexes $i\in I$ such that $B_i \not=0$.  
\end{dn}
\begin{md}
If $Y_1$, $Y_2$ are expressions of objects $X_s, s \in \Omega$ and $\varphi:Y_1 \rightarrow Y_2$ is the morphism built from morphisms $c,\Lh, \widehat{L}^A$, the identification and laws $\tx$, $\ts$ together with the composition, then they can be embedded into the following commutative diagram
\[\begin{CD}
Y_1 @>h_1>> \sum\limits_I A_i \\
@V \varphi VV @VV u  V\\
Y_2 @>h_2>> \sum\limits_I A_{\sigma(i)}
\end{CD}\]
where $(h_1,\sum\limits_I A_i)$, $(h_2,\sum\limits_I A_{\sigma(i)})$ are, respectively, canonical expansion forms of $Y_1$, $Y_2$; $\sigma$ is a permutation of the set $I$ and $u$ is an isomorphism  built from the morphism $c,$ the identification, the law $\ts$ together with the composition.
\end{md}
\begin{proof}
We can prove the proposition in case $\varphi$ is one of isomorphisms $c,\Lh, \widehat{L}^A$, $\Ah=id$, $g= d=id$, $a=id$, $l= r=id$. Next, we can prove it easily in case $\varphi$ is the sum $\tx$ or the product $\ts$ of two morphisms of above-mentioned ones.
\end{proof}
We now state the coherence thoerem.
\begin{dl}
Let $Y_1, Y_2, ..., Y_n$ be expressions of the family $X_s, s \in \Omega$ of objects in the quite strict Ann-category $\A.$ Let $\varphi_{i,i+1}:Y_i \rightarrow Y_{i+1}$ (i=1, 2, ..., n), $\varphi_{n,1}:Y_n \rightarrow Y_1$ be isomorphisms built from morphisms $c,\Lh, \widehat{L}^A$, identification by laws $\tx$, $\ts$ and the composition. Then, the following diagram 
\begin{center}
\setlength{\unitlength}{0.5cm}
\begin{picture}(10,5)

\put(4.3,1.3){$\varphi_{n,1}$}
\put(1,4.5){$\varphi_{1,2}$}
\put(3.4,4.5){$\varphi_{2,3}$}
\put(0,4){$Y_1$}
\put(2.6,4){$Y_2$}
\put(4.8,4){$Y_3$}
\put(7.2,4.2){...}
\put(9.6,4){$Y_n$}
\put(0,1){\line(1,0){10}}
\put(10,1){\line(0,1){2.5}}
\put(0,1){\line(0,1){2.5}}
\put(0.8,4.2){\line(1,0){1.2}}
\put(3.4,4.2){\line(1,0){1}}
\put(5.8,4.2){\line(1,0){1}}
\put(8,4.2){\line(1,0){1}}

\put(19,2){\emph{(3)}}
\end{picture}
\end{center}
commutes.
\end{dl}
\begin{proof}
Let $(h_i,\sum\limits_J A_{\sigma_i(j)})$ denote the canonical expansion form of $Y_i$. Consider the following diagram
\begin{center}
\setlength{\unitlength}{1cm}
\begin{picture}(15,6)
%{\scriptsize
\put(2,2){$\sum\limits_J A_{\sigma_1(j)}$}
\put(4.5,2){$\sum\limits_J A_{\sigma_2(j)}$}
\put(7.3,2){...}
\put(8.8,2){$\sum\limits_J A_{\sigma_{n-1}(j)}$}
\put(11.8,2){$\sum\limits_J A_{\sigma_n(j)}$}
\put(2.4,4){$Y_1$}
\put(4.8,4){$Y_2$}
\put(7.1,4){...}
\put(9.5,4){$Y_{n-1}$}
\put(12.1,4){$Y_n$}

\put(2,3){$h_1$}
\put(3.6,2.3){$u_{1,2}$}
\put(3.4,4.3){$\varphi_{1,2}$}
\put(5,3){$h_2$}
\put(7.3,0.8){$u_{n,1}$}
\put(7.3,5.2){$\varphi_{n,1}$}
\put(9,3){$h_{n-1}$}
\put(12.3,3){$h_n$}
\put(10.6,2.3){$u_{n-1,n}$}
\put(10.6,4.3){$\varphi_{n-1,n}$}

\put(2.5,5.5){\vector(0,-1){1}}
\put(2.5,3.8){\vector(0,-1){1.2}}
\put(2.5,0.5){\vector(0,1){1.2}}
\put(4.9,3.8){\vector(0,-1){1.2}}
\put(9.8,3.8){\vector(0,-1){1.2}}
\put(12.2,3.8){\vector(0,-1){1.2}}
\put(3.5,2.1){\vector(1,0){0.8}}
\put(6,2.1){\vector(1,0){1}}
\put(7.9,2.1){\vector(1,0){0.8}}
\put(10.6,2.1){\vector(1,0){1}}
\put(2.9,4.1){\vector(1,0){1.5}}
\put(5.3,4.1){\vector(1,0){1.5}}
\put(7.7,4.1){\vector(1,0){1.5}}
\put(10.4,4.1){\vector(1,0){1.5}}

\put(2.5,0.5){\line(1,0){9.6}}
\put(12.1,0.5){\line(0,1){1}}
\put(2.5,5.5){\line(1,0){9.6}}
\put(12.1,5.5){\line(0,-1){1}}
%}
\put(14.3,3){(4)}
\end{picture}
\end{center}
Then we have the diagram (2.2), where morphisms $u_{i,i+1}\ (i=1,2,...,n-1),\ u_{n,1}$ make regions (1), ..., (n) and the parameter commute according to Prop.2.2. They are built from $e,id$ and the laws $\tx,$ so according to the coherence theorem for a symmetric monoidal category, the region (b) commutes. Therefore, the region (a) commutes. This completes the proof.
\end{proof}
%%%%%%%Chuong 3%%%%%%%%%%%%%%%%%%%%
\section{The general case}
\noindent In this last section, we assume that $\A$, $\A'$ are Ann-categories with, respectively, constraints
\begin{align*}
(\Ah, c, (0, g, d), a, (1, l, r), \Lh,\Rh)\\
(\Ah', c', (0', g', d'), a', (1', l', r'), \Lh',\Rh')
\end{align*}
and $(F,\Breve{F},\Fm):\A \rightarrow \A'$ is a faithful Ann-functor such that the pair $(F,\Fm)$ is compatible with the unitivity constraints $(1, l, r)$, $(1', l', r')$. In addition, let $\Fh: F1\rightarrow 1'$ denote the isomorphism induced by the above compatibility.\\
Let $(X_i),i \in I$ be a non-empty, limited family of objects of $\A,$ and $Y=\Hh(X_i)$ be a certain expression of the family $(X_i),i \in I$. Then, the expression $Y'=\Hh(X_i')$ is called the {\it canonical image} of 
$Y=\Hh(X_i)$ under $F$ if
\begin{eqnarray*}
X_i'&=0'  \quad &\text{when } X_i=0 \\
X_i'&=1'  \quad &\text{when } X_i=1 \\
X_i'&=FX_i  \quad &\text{otherwise }  
\end{eqnarray*}
From this notion we give the following definition
\begin{dn}
We define a canonical isomorphism
$$f:FY=F(\Hh(X_i)) \rightarrow  F(\Hh(X_i')) $$
by induction on $Y$'s length as follows\\
1. If $Y$'s length is equal to 1, $Y=X_\alpha$ then
\begin{eqnarray*}
&f=\widehat{F}:F0 \rightarrow 0'   \quad &\text{in case  } X_\alpha=0 \\
&f=\Fh :F1 \rightarrow 1'   \quad &\text{in case  } X_\alpha=1 \\
&f =id: FX_\alpha \rightarrow FX_\alpha &\text{in other cases  }  
\end{eqnarray*}
2.  If $Y$'s length is more than 1, $Y=U_1 \tx U_2$ or $Y=U_1 \ts U_2.$ Then, the isomorphism $f$ is, respectively, the following compositions
\[\begin{CD}
FY=F(U_1 \tx U_2) @>\Fm>> FU_1\tx FU_2 @>f_1 \tx f_2 >> \Hh_1(X_i')\tx\Hh_2(X_i') 
\end{CD}\]
\[\begin{CD}
FY=F(U_1 \ts U_2) @>\breve{F}>> FU_1\ts FU_2 @>f_1 \ts f_2 >> \Hh_1(X_i')\ts\Hh_2(X_i') 
\end{CD}\]
where $f_1$, $f_2$ are canonical isomorphisms determined by inductive supposition.
\end {dn}
\begin{md}
Suppose that $\varphi:Y_1 \rightarrow Y_2$ is a morphism built from isomorphisms $\Ah,$ $c,$ $g,$ $d,$ $a,$ $l,$ $r,$ $\Lh,$ $\Rh,$ $\widehat{L}^A,$ $\widehat{R}^A$ in the Ann-category $\A$. Then, $\varphi$ can be embedded into the following commutative diagram
\[\begin{CD}
FY_1 @>f_1>> Y'_1 \\
@V F(\varphi)VV @VV\varphi'  V\\
FY_2 @>f_2>> Y'_2
\end{CD}\]
where $f_i$ are canonical isomorphisms corresponding to canonical images $Y'_i$ of $Y_i$ (i=1, 2), whereas $\varphi'$ is a morphism built from isomorphisms $\Ah',c',g',d',a',l',r',\Lh',\Rh', \widehat{L}^{A'},\widehat{R}^{A'}$ in the Ann-category $\A'.$
\end{md}
\begin{proof}
The proof is completely similar to the one of Proposition 2.2.
\end{proof}
Following is the main result of this paper.
\begin{dl}
Let $Y_1, Y_2,..., Y_n$ be expressions of the limited family of objects $(X_i)_{i \in I}$ of an Ann-category $\A.$ Let $\varphi_{i,i+1}:Y_i \rightarrow Y_{i+1}$ (i=1, 2, ..., n-1), $\varphi_{n,1}:Y_n \rightarrow Y_1$ be isomorphisms built from the isomorphisms $\Ah,c,g,d,a,l,r,\Lh,\Rh, \widehat{L}^A,\widehat{R}^A$, the identification and laws $\tx,\ts.$ Then, the diagram (2.1) commutes.
\end{dl}
\begin{proof}
From the theorem  2.4 [9], the Ann-category $\A$ can be faithfully embedded into a quite strict Ann-category $\A'$ by the faithful Ann-functor $(F,\Fc,\Fm)$. Moreover, $(F,\Fm)$ is compatible with the unitivity constraints. In order to prove the diagram (2.1) commutative, we consider its image under $F$\\
\begin{center}
\setlength{\unitlength}{1cm}
\begin{picture}(8,6)
\put(0,2.2){$Y'_1$}
\put(2,2.2){$Y'_2$}
\put(3.4,2.2){$...$}
\put(4.8,2.2){$Y'_{n-1}$}
\put(7.7,2.2){$Y_n$}
\put(0,4){$FY_1$}
\put(1.8,4){$FY_2$}
\put(3.4,4){...}
\put(4.8,4){$FY_{n-1}$}
\put(7.6,4){$FY_n$}
\put(-0.2,3.1){$f_1$}
\put(1.5,3.1){$f_2$}
\put(4.4,3.1){$f_{n-1}$}
\put(7.2,3.1){$f_n$}
\put(0.9,2.5){$\varphi'_{1,2}$}
\put(0.6,4.3){$F(\varphi_{1,2})$}
\put(6.2,2.5){$\varphi'_{n-1,n}$}
\put(5.8,4.3){$F(\varphi_{n-1,n})$}
\put(3.5,0.7){$\varphi'_{n,1}$}
\put(3.2,5.1){$F(\varphi_{n,1})$}

\put(0.2,5.5){\vector(0,-1){1}}
\put(0.2,3.8){\vector(0,-1){1.2}}
\put(0.2,0.5){\vector(0,1){1.5}}
\put(2.1,3.8){\vector(0,-1){1.2}}
\put(5.2,3.8){\vector(0,-1){1.2}}
\put(7.8,3.8){\vector(0,-1){1.2}}
\put(0.6,2.3){\vector(1,0){1.2}}
\put(2.5,2.3){\vector(1,0){0.7}}
\put(3.9,2.3){\vector(1,0){0.7}}
\put(5.6,2.3){\vector(1,0){2}}
\put(0.7,4.1){\vector(1,0){1}}
\put(2.4,4.1){\vector(1,0){0.8}}
\put(3.9,4.1){\vector(1,0){0.6}}
\put(5.9,4.1){\vector(1,0){1.5}}

\put(0.2,0.5){\line(1,0){7.6}}
\put(7.8,0.5){\line(0,1){1.5}}
\put(0.2,5.5){\line(1,0){7.6}}
\put(7.8,5.5){\line(0,-1){1}}

\end{picture}
\end{center}
where $f_i$ are canonical isomorphisms, whereas $\varphi_{i,i+1}$ ($i = 1, 2, ..., n - 1$), $\varphi_{n,1}$ are morphisms making regions from (1) to (n) and the parameter commute according to the Prop.3.2.  These morphisms are built from isomorphisms $c', \Lh',\widehat{L}^{A'},id$ and by laws $\tx,\ \ts.$ Applying Theorem 2.3, the region (b) commutes. This implies that the region (a) commutes. This completes the proof.
\end{proof}
\noindent \textit{\bf Remark. }The coherence theorem  can be stated in another way as follows: Between two objects of the Ann-category $\A,$ there exists no more than one morphism built from morphisms $\Ah,c,g,d,a,l,r,\Lh,\Rh, \widehat{L}^A,\widehat{R}^A$ and laws $\tx$, $\ts$. 
\begin{center}
\setlength{\unitlength}{1cm}
\begin{picture}(15,10)
{\scriptsize
\put(1.7,1){$(\sum\limits_{I'\times J} A_iB_j \ts \sum\limits_J A_\alpha B_j)(\sum\limits_T C_t)$}
\put(9,1){$\sum\limits_{I\times J\times T} A_i B_j C_t$}
\put(1.5,3.5){$(\sum\limits_{I'\times J} A_iB_j)(\sum\limits_T C_t)\ts (\sum\limits_J A_\alpha B_j)(\sum\limits_T C_t)$}
\put(8,3.5){$\sum\limits_{I'\times J\times T} A_i B_j C_t \ts \sum\limits_{J\times T} A_\alpha B_j C_t$}
\put(1.5,6.5){$(\sum\limits_I A_i)(\sum\limits_J B_j)(\sum\limits_T C_t)\ts A_\alpha (\sum\limits_J B_j)(\sum\limits_T C_t)$}
\put(8,6.5){$(\sum\limits_{I'} A_i)(\sum\limits_{J\times T} B_j C_t) \ts A_\alpha (\sum\limits_{J\times T} B_j C_t)$}
\put(2,9){$(\sum\limits_I A_i)(\sum\limits_J B_j)(\sum\limits_T C_t)$}
\put(9,9){$(\sum\limits_I A_i)(\sum\limits_{J\times T} B_j C_t)$}

\put(0.3,5){$F_{I,J}\tx id$}
\put(3.4,2.2){$id$}
\put(3.4,5.3){$(F_{I',J}\tx id)\ts (F_{A_\alpha}\tx id)$}
\put(3.4,7.8){$id$}
\put(6.5,1.4){$F_{I\times J, T}$}
\put(6.5,3.9){$ F_{I'\times J,T}\ts F_{J,T}$}
\put(5.5,6.9){$(id\tx F_{J,T})\ts (id\tx F_{J,T})$}
\put(6,9.3){$id\tx F_{J,T}$}
\put(10.2,2.2){$id$}
\put(10.2,5.3){$F_{I',J\times T}\ts f_{A_\alpha}$}
\put(10.2,7.8){$id$}
\put(13.8,5){$F_{I, J\times T}$}

\put(3.2,3){\vector(0,-1){1.5}}
\put(3.2,6.2){\vector(0,-1){2}}
\put(3.2,8.6){\vector(0,-1){1.7}}
\put(10,3){\vector(0,-1){1.5}}
\put(10,6.2){\vector(0,-1){2}}
\put(10,8.6){\vector(0,-1){1.7}}
\put(6,1.1){\vector(1,0){2.5}}
\put(6.7,3.6){\vector(1,0){1.5}}
\put(7,6.6){\vector(1,0){0.8}}
\put(5,9.1){\vector(1,0){3.6}}

\put(0,1.1){\line(1,0){1.6}}
\put(0,1.1){\line(0,1){8}}
\put(0,9.1){\line(1,0){1.8}}
\put(15,9.1){\line(-1,0){3.2}}
\put(15,1.1){\line(0,1){8}}
\put(15,1.1){\line(-1,0){3.8}}
}
\end{picture}
\end{center}
\newpage
\begin{center}

\end{center}
Math. Dept.,Hanoi University of Education.\\
E-mail address: nguyenquang272002@gmail.com
\end{document}